\documentclass[11pt]{article}
\usepackage{amsmath}
\usepackage{amssymb, latexsym, amsfonts, amscd, amsthm, mathrsfs, enumerate, esint}
\usepackage[usenames,dvipsnames]{color}
\usepackage[all]{xy}
\usepackage{graphicx}
\usepackage{epsfig}
\usepackage{hyperref}
\usepackage{marginnote}
\usepackage{mathtools} 			
\usepackage{bbm}
\usepackage{amssymb}
\usepackage{amsmath}
\usepackage{amsfonts,latexsym,amstext}
\usepackage{marginnote}
\usepackage[numbers,sort]{natbib}
\usepackage{stmaryrd}
\usepackage{soul}

\numberwithin{equation}{section}

\definecolor{purple}{rgb}{0.9,0,0.8}

\definecolor{gray}{rgb}{0.7,0.7,0.7}

\newtheorem{theorem}{Theorem}

\newtheorem{proposition}{Proposition}

\newcommand{\G}{{\mathcal{G}}}
\newcommand{\Ed}{{\mathcal{E}}}
\newcommand{\V}{{\mathcal{V}}}

\usepackage{hyperref}
\newcommand{\footremember}[2]{%
    \footnote{#2}
    \newcounter{#1}
    \setcounter{#1}{\value{footnote}}%
}
\newcommand{\footrecall}[1]{%
    \footnotemark[\value{#1}]%
}

\title{Stability of JSQ in queues with general server-job class compatibilities}

\author{%
James Cruise\footremember{HW}{Heriot-Watt University}%
\and Matthieu Jonckheere\footremember{UBA}{University of Buenos Aires} %
 \and Seva Shneer\footrecall{HW}%
  }

\begin{document}
\maketitle

\abstract{
We consider Poisson streams of exponentially distributed jobs arriving at each edge of a hypergraph of queues.
Upon arrival, an incoming job is rooted to the shortest queue among the corresponding vertices. This generalizes many known models such as power-of-d load balancing
and JSQ (join the shortest queue) on generic graphs. 
 We provide a generic condition for stability of this model.
We show that some graph topologies lead to a loss of capacity, implying more restrictive stability conditions than in, e.g., complete graphs.

\maketitle

\section{Introduction}
Load balancing schemes involving various types of load information have become absolutely essential for the quality of service in many important modern areas of applications: call centers, commercial server farms, scientific computing,
vehicles systems, data centres, and others.


In the last decades, many theoretical results have focused on the case
of parallel servers (i.e. the complete graph model in our setting, see below).
In particular, a lot of attention has been given to mean-field type results for 
complete graphs models, and schemes like join the shortest queue (JSQ) among $n$ available or join the shortest of $d$ among $n$ queues (JSQ(d)), where $n$ is large, started with the seminal work of \cite{mitzenmacher,vvedenskaya96} and complemented by several papers. (This last scheduling is also known as power-of-d).
Transient functional law of large numbers and propagation of chaos for JSQ and JSQ(d) have been obtained for instance in \cite{graham00,Mukho15} for FIFO scheduling.
For general service time distributions, the results are scarcer. 
 For service time distributions with decreasing hazard rate and FIFO service discipline, propagation of chaos properties and asymptotic behaviour of the number of occupied servers were obtained for the JSQ(d) 
policy in \cite{bramson2012}.
 In \cite{foss2016}, the convergence of the mean-field limit of the join-the-idle-queue (JIQ) policy in the stationary regime was proved under light traffic conditions.

More recently, \cite{stolyar2015} obtains mean-field limit for the JIQ, and \cite{mukherjee2016} computes the diffusive limit in the Halfin-Whitt regime for a class of policies of which JIQ and JSQ(d) policies are special cases.  Interestingly they show that JIQ is optimal at this diffusive scale. For the JSQ policy, the large-server heavy-traffic limit was derived in \cite{Eschenfeldt2015}. Different scalings (asymptotic relations between number of servers, loads, buffer sizes) were considered for instance in \cite{jonckheereprabhu18}.

We consider here a spatial generalization of these models
where independent Poisson streams arrive to edges of a hypergraph
and are routed (either statically or dynamically) to one of the vertices
corresponding to the edge.
The well-known case of JSQ(d) can be retrieved simply by considering
a hypergraph where every set of $d$ vertices forms an edge (i.e., a complete $d$-hypergraph).

It is also closely related to a model introduced for instance in \cite{mukherjee2018asymptotically}. 
In their context, each node has an associated arrival process and upon arrival a customer associated with vertex $i$ considers the queue lengths of $i$ and at all its neighbours. It is then allocated to be served by the server at a vertex with the smallest queue length from all those examined. 
This model is also a special case of the one considered here, as we can construct a hypergraph 
with $n$ edges, each edge representing a node in the original graph and all the neighbours of that node in the original graph. Then, arrivals on that edge are equivalent to arrivals at the associated node in the original graph. 
Note that a further extension to this model is considered in \cite{budhiraja2017supermarket}, where a random subset of the neighbours is considered when making the decision about routing rather than the full neighbourhood. Again we can create an equivalent hypergraph model by introducing an edge for each random combination which might be considered.

In the models of \cite{mukherjee2018asymptotically} and \cite{budhiraja2017supermarket} stability is trivial as the state of the system is always dominated, at least in the sense of the maximal queue size, by that in a system where arrivals at a specific node have to be served by the server at that node. The authors are thus interested in the occupancy measures in some growing and/or random topologies. In our more general model stability is non-trivial and we thus study it, in fixed topologies.

It is also worth noting that our model is equivalent to
a bipartite graph between customers arrival processes and servers (i.e. customer arrival processes connected by edges to the servers that can serve the arriving jobs). These models have been referred to as skill-based systems and studied by several papers (see, e.g., \cite{AdanWeiss1}, \cite{AdanWeiss2} and references therein.
The load balancing studied in the aforementioned papers however is join the shortest workload (JSW), which has (for a specific state description) a product form stationary measure, and thus stability condition is easily derived. This is not the case for JSQ networks. Furthermore, even if the JSW discipline has been shown to have the largest stability region (and, in particular, larger than that of the JSQ discipline) in highly symmetric models, like multi-server queues or networks of these queues, (see \cite{Foss1,Foss2}), it is not clear that it can be generalized to asymmetric networks topologies.

\

%
%
%
%
%
%
%
%

The article is organized as follows.
In Section \ref{sec:model}, we precisely define our load balancing model on a hypergraph.
In Section \ref{sec:stabstat}, we derive the stability conditions of static allocations and show that they do not necessarily correspond to a trivial rate conservation condition.
In Section \ref{sec:stabdyn}, we present the main contribution of our paper. Namely, we show that the JSQ load balancing on a hypergraph ensures stability if and only if there is a stable static allocation.

\section{Model}\label{sec:model}

Let $\G= (\V,\Ed)$ be a hypergraph with vertices $\V$ and edges $\Ed$, where each $e \in \Ed$ is a subset of $\V$. Associated with each edge $e \in \Ed$ is a class of customers who arrive as a Poisson process with rate $\lambda_e$. Associated with every vertex $v \in \V$ is a single server, and we denote its queue size at time $t$ by $Q_v(t)$. Each customer served by the server at vertex $v$, irrespective of its arrival class, requires an exponential service with rate $\mu_v$. Customers in the class associated with an edge $e \in \Ed$ can be served by any of the servers at vertices incident with the edge. In other words, customers in class in $e \in \Ed$ can be served by any server at a vertex $v \in e$. Upon arrival, a customer is allocated to a server and joins the relevant queue. The customers in each queue are served using the FIFO discipline. 


We now introduce static and dynamic allocation policies which we analyse.

\subsection{Static allocation}

Associated with each edge $e \in \Ed$ there are probabilities $\{p_{v,e}\}_{v \in e}$ such that $p_{v,e} \ge 0$ for all $v \in e$ and $$\sum_{v \in e} p_{v,e}=1.$$
When a customer arrives upon an edge $e \in \Ed$, it is allocated to a node $v \in e$ with probability $p_{v,e}$, independently of all other arrivals and services. Let $\mathbf{P} = \{\{p_{v,e}\}_{v \in e}\}_{e \in \Ed}$ refer to a given allocation for each edge. 

Therefore the total arrival process at node $v$ is a Poisson process, independent of all other nodes, of rate $$\lambda_v(\mathbf{P})=\sum_{e \in \Ed(v)} p_{v,e} \lambda_e,$$
where $\Ed(v) = \{e \in \Ed: v \in e\}$ is the set of all edges containing node $v$. 

\subsection{Dynamic allocation}

The dynamic allocation aims to load balance across the network by utilising join-the-shortest-queue dynamics. 

Upon an arrival of a customer on an edge $e \in \Ed$, the queue sizes at all nodes $v \in e$ are examined, and the customer is routed to the shortest of these. If there are more than one queues with the smallest size, the customer is routed to any one of them, at random with equal likelihoods.

Note this is a natural definition of join-the-shortest-queue in this setting. 

\section{Stability of static allocations}\label{sec:stabstat}

For static allocations, the queues decouple and the stability condition is straightforward.

\begin{proposition}
A static allocation is stable if and only if
$$  \sum_{e \in \Ed(v)} p_{v,e} \lambda_e < \mu_v$$
for all $v \in \V$, or, alternatively,
$$
\max_{v \in \V} \frac{1}{\mu_v} \sum_{e \in \Ed(v)} p_{v,e} \lambda_e < 1.
$$
\end{proposition}

\begin{proof}
For the static allocation we know that for a given allocation $\mathbf{P}$ the arrival process at each vertex is an independent Poisson process with arrival rate $$\lambda_v(\mathbf{P})=\sum_{e \in \Ed(v)} p_{v,e} \lambda_e.$$
From this we have the stability condition associated with each node is $$\sum_{e \in \Ed(v)} p_{v,e} \lambda_e < \mu_v,$$
so that the stability condition for the whole system follows.
\end{proof}

While the previous result concerns a single possible allocation, we now consider the best possible allocation and the maximal stability region of the graph. 
The stability region for a given graph is maximized by minimizing over the possible allocations, as shown in the following.

\begin{proposition}
There exists a stable static allocation if and only if
$$\min_{\mathbf{P}}\left(\max_{v \in \V} \frac{1}{\mu_v} \sum_{e \in \Ed(v)} p_{v,e} \lambda_e\right) < 1.$$
\end{proposition}

\subsection{Particular case: symmetric system}

In this subsection we consider a particular case of our general setting where all customer classes have the same arrival intensity and all jobs require service times with the same distribution. More precisely, $\lambda_e = \lambda$ for all $e \in \Ed$ and $\mu_v = \mu$ for all $v \in \V$. The general stability condition thus reduces to the requirement
$$
\lambda\left(\min_{\mathbf{P}}\left(\max_{v \in \V} \sum_{e \in \Ed(v)} p_{v,e}\right)\right) < \mu.
$$

Note that practically, the maximal arrival rate characterizing the optimal static stability condition in this case can be computed as
$\lambda^* = {\frac{\mu}{z^*}}$ where $z^*$ is
 the solution of the following linear program:
\begin{align*}
& \min_{p,z}  z,\\
&\sum_{e \in \Ed(v)} p_{v,e} \le z, \forall v \in \V,\\
 & 0 \le p_{e,v},\\
 & \sum_{v \in e} p_{e,v}= 1, \forall e \in \Ed.
\end{align*}

\

It is worth noting that if there exists an allocation $P$ which equalizes the $\left(\sum_{e \in \Ed(v)} p_{v,e}\right)$ over all vertices $v$, then there is no loss in stability region due to the restrictions imposed by the graph structure, i.e. the maximum possible total arrival rate into the network is equal to the total service rate of the network. Indeed, as all the values of $\left(\sum_{e \in \Ed(v)} p_{v,e}\right)$ are equal, each of them is necessarily equal to $|\V|/|\Ed|$, where $|\Ed|$ is the number of edges and $|\V|$ is the number of vertices. The stability condition hence reads $\lambda |\Ed| < \mu |\V|$, which is exactly the requirement that the total arrival rate is smaller than the total service rate. In this case we obtain complete resource pooling in the sense of stability (but possibly in a weaker sense than state space collapse).

An interesting question thus arises: can we understand what properties of the graph enable us to find a balanced (i.e. maximal stable in terms of rate conservation) allocation, and when it is not possible?
We partially answer this question in the next section.

\subsection{Addition of edges can lead to smaller stability region}

To better understand the question posed above, we provide three revealing examples: firstly two extreme cases where balance is always achievable and then an example where balance is not achievable and we do observe a loss of capacity. We focus on standard graphs in this section.

Let us consider two extreme graphs on $n$ vertices, the circle and the complete graph. In both cases the allocation of $(1/2, 1/2)$ on every edge balances the loads and enables the maximum stability regions in these cases.

For an example where balance can not be obtained, consider a graph containing $2k$ vertices for $k>2$ and separate them into two groups of $k$ vertices. The first $k$ vertices form a clique. The remaining $k$ vertices are then leaves in a graph connected to a single node in the clique and each node in the clique is connected to a single leaf. It is not difficult to see that the best allocation you can achieve here is to equalize across the clique and then on all leaf edges, to allocate all the traffic to the leaf vertex. This gives the following pair of stability constraints:
$$ \lambda < \mu \textrm{ and } \frac{k-1}{2} \lambda < \mu.$$
Note that for $k>3$ the first condition is superfluous, so the maximum arrival rate per edge is $\frac{2\mu}{k-1}$. Since there are $k(k+1)/2$ edges in this graph the maximum total stable arrival rate is 
$$ \frac{k(k+1)}{(k-1)} \mu,$$
which is substantially below the total service rate of $2k \mu$.

\



\section{Stability of dynamic allocations} \label{sec:stabdyn}

We now turn to our main result which characterizes the stability of dynamic allocations in terms of the maximal static stability condition.

\begin{theorem}
The dynamic allocation is stable if and only if the maximal static allocation is stable, i.e.,
$$\min_{\mathbf{P}}\left(\max_{v \in \V} \frac{1}{\mu_v} \sum_{e \in \Ed(v)} p_{v,e} \lambda_e\right) < 1.$$
\end{theorem}

\begin{proof}

\

{\it Necessary condition.}


 Suppose the dynamic allocation is stable. Then there exist stationary probabilities, $\pi_{v,e}$ say, for a customer arriving at edge $e$ to be routed to vertex $v \in e$. Since the network is stationary, rate stability implies that
$$ \lambda \sum_{e \in \in \Ed(v)}\pi_{v,e} =  \mu_v P(X_v >0).$$
In particular:
$$ \lambda \sum_{e \in \in \Ed(v)}\pi_{v,e} <  \mu_v.$$
 Hence the collection $\{\{\pi_{v,e}\}_{v \in e}\}_{e \in \Ed}$ clearly forms a stable static allocation.

\

\

{\it Sufficient condition.}

Assume there exist a stable static allocation $\mathbf{P}$:
$$
\sum_{e \in \Ed(v)} p_{v,e} \lambda_e < \mu_v
$$
for all $v \in \V$. As there is a finite number of vertices, fix $\varepsilon > 0$ such that
\begin{equation} \label{eq:gap}
\sum_{e \in \Ed(v)} p_{v,e} \lambda_e - \mu_v < -\varepsilon
\end{equation}
for all $v \in \V$.

Consider now the system with dynamic allocations (with arbitrary tie breaks) and consider the Lyapunov function
$$
L(\bar{x}) = \sum_{v \in V} x_v^2.
$$
We know that at rate $\mu_v$ there is a departure from node $v$, and at rate $\lambda_e$ there is an arrival at edge $e$, which will then go to the minimal adjacent queue. Therefore, conditioned on the current queue lengths being $(x_v)_{v \in \V}$, the drift of the Lyapunov function is equal to
$$
- 2 \sum_{v \in \V} \mu_v x_v + 2 \sum_e \lambda_e \min_{v \in e} x_v + c,
$$
where $c$ is a finite constant. Note now that
$$
\min_{v \in e} x_v \le \sum_{v \in e} p_{v,e} x_v,
$$
as $\sum_{v \in e} p_{v,e}=1$. The drift can then be bounded from above by
\begin{align*}
& - 2 \sum_{v \in \V} \mu_v x_v + 2 \sum_e \lambda_e \sum_{v \in e} p_{v,e} x_v + c = - 2 \sum_{v \in \V} \mu_v x_v + 2 \sum_{v \in \V} \sum_{e \in \Ed(v)}  \lambda_e p_{v,e} x_v + c \\
& =2 \sum_{v \in \V} x_v (- \mu_v + \sum_{e \in \Ed_v} \lambda_e p_{v,e}) + c \le - 2 \varepsilon \sum_{v \in \V} x_v + c,
\end{align*}
thanks to \eqref{eq:gap}. The drift is therefore smaller than $-\delta < 0$ as long as $\sum_{v \in \V} x_v > (c+\delta)/(2 \varepsilon)$, which is sufficient for stability.

\end{proof}

\section{Conclusion}
We provided necessary and sufficient conditions of stability for a model of load balancing on fixed hypergraphs that generalize most previous models in the literature.
Interesting and difficult challenges consist in characterizing these conditions for large classes of random graphs.

\section*{Acknowledgements}
The authors are grateful to the associate editor for their careful reading of the paper and useful comments and suggestions, especially for brining to our attention relevant work on skill-based routing.

\bibliographystyle{plain}
\bibliography{TheBib}

\end{document}